\documentclass{article}
\usepackage{amsmath}
\usepackage{amsfonts}
\usepackage{amsthm}
\usepackage{amssymb}
\usepackage[english]{babel}
\usepackage[ansinew]{inputenc}
\usepackage[all]{xy}
\usepackage{color}

\theoremstyle{definition}

\newtheorem{defi}{Definition}[section]
\newtheorem{remark}[defi]{Remark}
\newtheorem{notation}[defi]{Notation}

\theoremstyle{plain}

\newtheorem{theorem}[defi]{Theorem}

\newtheorem{lemma}[defi]{Lemma}

\author{Marco Antei, Vikram B. Mehta}

\title{On the Grothendieck-Lefschetz Theorem for a Family of Varieties}
\begin{document}
\maketitle

\noindent \textbf{Abstract.} Let $k$ be an algebraically closed field of characteristic $p>0$, $W$ the ring of Witt vectors over $k$ and ${R}$ the integral closure of $W$ in the algebraic closure ${\overline{K}}$ of $K:=Frac(W)$; let moreover $X$ be a smooth, connected and projective scheme over $W$  and $H$ a relatively very ample line bundle over $X$. We prove that when $dim(X/{W})\geq 2$  there exists an integer $d_0$, depending only on $X$, such that for any $d\geq d_0$, any $Y\in |H^{\otimes d}|$ connected and smooth over ${W}$ and any $y\in Y({W})$ the natural ${R}$-morphism of fundamental group schemes $\pi_1(Y_R,y_R)\to \pi_1(X_R,y_R)$  is faithfully flat,  $X_R$, $Y_R$, $y_R$ being respectively the pull back of $X$, $Y$, $y$ over $Spec(R)$.  If moreover  $dim(X/{W})\geq 3$ then there exists an integer $d_1$, depending only on $X$, such that for any $d\geq d_1$, any $Y\in |H^{\otimes d}|$ connected and smooth over ${W}$ and any section $y\in Y({W})$ the morphism $\pi_1(Y_R,y_R)\to \pi_1(X_R,y_R)$ is an isomorphism. 

\tableofcontents
\bigskip

\section{Introduction}\label{sez:Intro} 

The notion of fundamental group scheme of a connected and reduced scheme over a perfect field has been introduced by Madhav Nori in \cite{Nor} and \cite{Nor2} as the affine group scheme over $k$ naturally associated to the tannakian category of essentially finite vector bundles. Then in \cite{Gas} it has been generalized by Gasbarri for integral schemes over a connected Dedekind scheme. In this latter  description, however,  the tannakian tool used by Nori is absent. This has been added  by  the second (alphabetically) author and Subramanian in  \cite{MS2} where they describe  the fundamental group scheme of a  smooth and projective scheme over a certain Pr\"ufer ring $R$ (more details will be recalled in section \ref{sec:Theorems}) by means of tannakian lattices introduced by Wedhorn in \cite{Wed}.\\

Now let $Z$ be a smooth and projective variety over an algebraically closed field $k$. If $Y$ is any smooth ample hypersurface on $Z$  and $y$ any point of $Y$  then by Grothendieck-Lefschetz theory (cf. \cite{SGA2}, Exposé X) we know that the induced group homomorphism between the \'etale fundamental groups $$\pi_1^{\text{\'et}}(Y,y)\to \pi_1^{\text{\'et}}(Z,y)$$ is surjective when $dim(Z)\geq 2$ and an isomorphism when $dim(Z)\geq 3$, thus in particular if $char(k)=0$ the same result automatically holds for the fundamental group scheme (even when $k$ is not algebraically closed and this can be seen using the fundamental short exact sequence). When $char(k)=p>0$ then in \cite{Mehta1} and \cite{BiswHoll} it has been proved, independently, that theorems of  Grothendieck-Lefschetz type hold in the following formulation: let $H$ be a very ample line bundle over $Z$ then when $dim(Z) \geq 2$ the natural homomorphism $\widehat{\varphi}:\pi_1(Y,y)\to \pi_1(Z,y)$  between fundamental group schemes induced by the inclusion map $\varphi:Y\hookrightarrow Z$ is faithfully flat  whenever $Y$ is in the complete linear system $|H^{\otimes d}|$ for any integer $d\geq d_0$ where  $d_0$ is an integer depending only on $Z$.  If moreover $dim(Z) \geq 3$, then $\widehat{\varphi}$ is an isomorphism whenever $Y$ is in the complete linear system $| H^{\otimes d}|$ for any integer $d\geq d_1$ where  $d_1$ is an integer depending only on $Z$.  \\

Let finally $k$ be an algebraically closed field of characteristic $p>0$, $W$ the ring of Witt vectors over $k$ and ${R}$ the integral closure of $W$ in the algebraic closure ${\overline{K}}$ of $K:=Frac(W)$, in this paper we prove the following generalization (where the subscript ${}_R$ will denote the pull back over $Spec(R)$):

\begin{theorem}(Cf. Theorems \ref{theorem1} and \ref{theorem2}) Let  ${X}$ be a smooth and projective scheme over ${W}$  and  $H$ a relatively very ample line bundle over $X$. \begin{enumerate}

\item If  $dim(X/{W})\geq 2$ then there exists an integer $d_0$ (depending only on $X$) such that for any $d\geq d_0$, any $Y\in |H^{\otimes d}|$ connected and smooth over ${W}$  and any $y\in Y({W})$ the natural ${R}$-morphism of fundamental group schemes $\pi_1(Y_R,y_R)\to \pi_1(X_R,y_R)$  is faithfully flat. 

\item If moreover  $dim(X/{W})\geq 3$ then there exists an integer $d_1$ (depending only on $X$) such that for any $d\geq d_1$, any $Y\in |H^{\otimes d}|$ connected and smooth over ${W}$ and any $y\in Y({W})$ the  morphism $\pi_1(Y_R,y_R)\to \pi_1(X_R,y_R)$ is an isomorphism. 
\end{enumerate}
\end{theorem}

\indent \textbf{Acknowledgments:} M. Antei would like to thank Vikram Mehta for the invitation to the Tata Institute of Fundamental Research of Mumbai where this paper has been conceived. V. B. Mehta would like to thank the ICTP, Trieste for hospitality.

\section{Vanishing lemmas}

\begin{notation}Throughout this section $k$ will be an algebraically closed field of positive characteristic $p$ and $X$ a smooth and projective scheme over $k$.
\end{notation}

In  section \ref{sec:Theorems} we will strongly make use of the   vanishing Lemmas \ref{lemmaVik1} and \ref{lemmaVik2} whose idea  is already contained in \cite{Mehta1}, lemmas 2.2 and 3.3. Let $H$ be a very ample line bundle on $X$. Moreover let $F_X:X\to X$ be the absolute Frobenius morphism, then we denote by $F_X^{m}$ the $m$-th iterate of $F_X$. We recall that a vector bundle $V$ over $X$  is essentially finite (cf. \cite{Nor} and \cite{Nor2} for Nori's definition) if there exists a finite $k$-group scheme $G$  and a principal $G$-bundle $\pi:E\to X$ such that $\pi^{\ast}(V)$ is trivial on $E$. We also recall (cf. \cite{MS1}) that a vector bundle $V$ over $X$ is called $F$-trivial if there exists an integer $n\geq 0$ such that $F_X^{n \ast}(V)$ is trivial on $X$.  When $V$ is essentially finite  then there exists an integer $n\geq 0$ such that $F_X^{n \ast}(V)$ is Galois \'etale trivial (cf. \cite{MS1}), i.e. there exists  a Galois \'etale covering $\pi':E'\to X$ such that  $\pi'^{\ast}(F_X^{n \ast}(V))$ is trivial on $E'$ (the converse is also true). We denote  by $EF(X)$ the tannakian category of essentially finite vector bundles over $X$. We finally recall that $V\in EF(X)$ if and only if the dual $V^{\ast}\in EF(X)$. We will often use the notation $V^{(p^m)}=F_X^{m \ast}(V)$ for every integer $m\geq 0$ for the comfort of the reader. 

\begin{remark}\label{remCartier} Let $\Omega_X^{\bullet}$ be the De Rham complex and let us set $B^j_X:=Im(d:\Omega_X^{j-1}\to \Omega_X^{j})$ and $Z^j_X:=Ker(d:\Omega_X^{j}\to \Omega_X^{j+1})$ for all $j>0$, then we can define the Cartier operator $$C_X:Z^{\bullet}_X\to \Omega_X^{\bullet}$$ whose kernel is $B^{\bullet}_X$; in particular we have the following exact sequences of vector bundles
\begin{equation}\label{eqCar1}0\to B^{1}_X\to Z^{1}_X \to \Omega_X^{1}\to 0,\end{equation}
\begin{equation}\label{eqCar2}0\to \mathcal{O}_X \to {F_X}_{\ast}(\mathcal{O}_X)\to B^1_X\to 0,\end{equation}
\begin{equation}\label{eqCar3}0\to Z^{1}_X \to {F_X}_{\ast}(\Omega_X^{1})\to B^{2}_X\to 0.\end{equation}
\end{remark}

\begin{lemma}\label{lemmaVik1} Assume $dim(X)\geq 2$. Let $V$ be any essentially finite vector bundle over $X$, then there exists a \textit{uniform} (i.e. depending only on $X$) positive integer $n_0$ such that  $H^1(X,V(-n))$  vanishes for all $n> n_0$.
\end{lemma}

\proof
 (See also \cite{BiswHoll}, Lemma 4.7.) Let us  consider the exact sequence (\ref{eqCar2}):
$$
0\to \mathcal{O}_X \to {F_X}_{\ast}(\mathcal{O}_X)\to B^1_X\to 0
$$
and tensor it by $V(-n)$; then we have  the induced long exact sequence

$$..\to  H^0(X,V(-n)\otimes B^1_X)\to  H^1(X,V(-n))\to  H^1(X,{F_X}_{\ast}(\mathcal{O}_X)\otimes V(-n))\to ..$$
By the projecton formula we have ${F_X}_{\ast}(\mathcal{O}_X)\otimes V(-n)\simeq {F_X}_{\ast}({F_X}^{\ast}(V(-n)))$ thus $H^1(X,{F_X}_{\ast}(\mathcal{O}_X)\otimes V(-n))\simeq H^1(X,{F_X}^{\ast}(V(-n)))= H^1(X,V^p(-np))$. Moreover $$H^0(X,V(-n)\otimes B^1_X)\simeq \mathcal{H}om(V^{\ast}(n),B^1_X)(X)\simeq Hom(V^{\ast}(n),B^1_X)$$
and $Hom(V^{\ast}(n),B^1_X)=0$ as soon as  $\mu(V^{\ast}(n))> \mu_{max}(B^1_X)$. Set $n_0:=\mu_{max}(B^1_X)$, so that previous inequality becomes $n>n_0$. Thus, when $n>n_0$, the morphism 
$$H^1(X,V(-n))\to  H^1(X,V^p(-np))$$
becomes injective. Iterating the process (tensor (\ref{eqCar2}) by $V^p(-np)$ and so on) we obtain the injective morphism

\begin{equation}\label{eqINJ}
H^1(X,V(-n))\to  H^1(X,V^{(p^m)}(-np^m))
\end{equation}

for every $m\geq 0$ and every $n> n_0$ where, we recall, $n_0=\mu_{max}(B^1_X)$ depends only on $X$. We have already recalled that since $V\in EF(X)$ then  there exists an integer $l\geq 0$  such that  $F_X^{l \ast}(V)$ is Galois \'etale trivial, then so are its stable components $V_i, i=1, .., L$, where $L$ is the length of its Jordan-H\"older filtrations. Thus every $V_i$ is stable and Galois \'etale trivial then by \cite{Duc}, Th\'eor\`eme 2.3.2.4 for every $i$ there exists an integer $t_i> 0$ such that $F_X^{t_i \ast}(V_i)\simeq V_i$. This implies that the isomorphism classes of stable components of the vector bundles in the family $\{F_X^{t \ast}(V)\}_{t>0}$ are only finitely many. Denote these isomorphism classes of stable vector bundles by $\{W_j\}_{j\in J}$, then by the Enriques-Severi-Zariski-Serre vanishing lemma  and the fact that $|J|<+\infty$ we know that  $H^1(X,W_j(-np^s))=0$ for all $s>>0$ then $H^1(X,V^{(p^m)}(-np^m))$ for all $m>>0$. Now remember that for $n>n_0$ we have  the injection (\ref{eqINJ}) hence  $H^1(X,V(-n))=0$.
\endproof

\begin{remark}\label{remDecember}
Using arguments similar to those that in Lemma \ref{lemmaVik1} allowed us to prove that $H^1(X,V^{(p^m)}(-np^m))$ vanishes for all $m>>0$ one also proves  that  $ H^2(X,V^{(p^t)}(-np^t))$ vanishes for all $t >>0$.
\end{remark}

The aim of the reminder of this section is to prove a vanishing result for the group $H^2(X,V(-n))$. This will be done in Lemma \ref{lemmaVik2}, but we first need some preliminary steps. Let us start with a lemma which is also proved in  \cite{BiswHoll}, Proposition 4.11:

\begin{lemma}\label{lemmaOmega} Assume $dim(X)\geq 2$. Let $V$ be any essentially finite vector bundle over $X$, then there exists a \textit{uniform}  positive integer $n_1$ such that   $H^1(X,\Omega_X^1\otimes V(-n))$ vanishes for all  $n> n_1$.
\end{lemma}

\proof
 Let $\mathcal{T}_X$ be the tangent bundle of $X$, then $\mathcal{T}_X(s)$ is generated by a finite number of global sections  for some integer $s$ depending only on $X$, then we have an exact sequence

\begin{equation}
0\to S^{\ast}\to \mathcal{O}_X^N\to \mathcal{T}_X(s)\to 0
\end{equation}
for some vector bundle $S$; dualizing we obtain 

\begin{equation}
0\to \Omega^1_X(-s)\to \mathcal{O}_X^N\to S\to 0
\end{equation}
then we tensor by $V(-m)$ for some positive integer $m$ and we get the following exact sequence

\begin{equation}\label{eqOmega}
0\to \Omega^1_X(-s)\otimes V(-m)\to \mathcal{O}_X^N\otimes V(-m)\to S \otimes V(-m)\to 0;
\end{equation}
consider the induced long exact sequence

\begin{equation}\label{eqOmegaLong}
H^0(X,S \otimes V(-m))\to H^1(X,\Omega^1_X\otimes V(-m-s))\to H^1(X, V(-m))^N 
\end{equation}
Now $H^0(X,S \otimes V(-m))=0$ as soon as $m>\mu_{max}(S)$ and $s_0:=\mu_{max}(S)$ is clearly independent from $V$. Moreover by Lemma \ref{lemmaVik1} there exists $n_0$ dependent only on $X$ such that $H^1(X, V(-m))^N=0$ for all $m>n_0$. Thus $H^1(X,\Omega^1_X\otimes V(-m-s))=0$ for $m>max\{n_0,s_0\}$.  We have thus proved that there exists an integer $n_1$ such that $H^1(X,\Omega^1_X\otimes V(-n))=0$ for all $n>n_1$.
\endproof

\begin{remark}\label{remOmega}  Assume $dim(X)\geq 2$.  Let $V$ be any essentially finite vector bundle over $X$, $n$ an integer such that  $n> n_1$, as defined in Lemma \ref{lemmaOmega} then $H^1(X,{F_X}_{\ast}(\Omega_X^1)\otimes V(-n))$ vanishes. Indeed ${F_X}_{\ast}(\Omega_X^1)\otimes V(-n)\simeq {F_X}_{\ast}((V^p(-np))\otimes \Omega_X^1)$ thus $H^1(X,{F_X}_{\ast}(\Omega_X^1)\otimes V(-n))\simeq H^1(X,(V^p(-np))\otimes \Omega_X^1)$ and the latter is  trivial by Lemma \ref{lemmaOmega}, $V^p$ being essentially finite.
\end{remark}

\begin{lemma}\label{lemmaPreVik2}  Assume $dim(X)\geq 2$. Let $V$ be any essentially finite vector bundle over $X$, then there exists a \textit{uniform}  positive integer $n_2$ such that   $H^1(X,V(-n)\otimes B^1_X)$ vanishes for all  $n> n_2$.
\end{lemma}
\begin{proof}
We tensor (cf. Remark \ref{remCartier})

$$0\to B^{1}_X\to Z^{1}_X \to \Omega_X^{1}\to 0$$
by $V(-n)$ then we get the long exact sequence
$$..\to H^0(X,\Omega_X^{1} \otimes V(-n))\to H^1(X,B^{1}_X\otimes V(-n))\to H^1(X,Z^{1}_X\otimes V(-n))\to .. $$
but $H^0(X,\Omega_X^{1} \otimes V(-n))=Hom(V^*(n),\Omega_X^{1})=0$ as soon as $n>\mu_{max}(\Omega_X^{1})$ where $r_1:=\mu_{max}(\Omega_X^{1})$ is independent of $V$. Thus 

\begin{equation}\label{Eqinj1} H^1(X,B^{1}_X\otimes V(-n))\hookrightarrow H^1(X,Z^{1}_X\otimes V(-n))\end{equation}
is injective for all $n>r_1$. In a similar way, tensoring 

$$0\to Z^{1}_X \to {F_X}_{\ast}(\Omega_X^{1})\to B^{2}_X\to 0$$ 
by $V(-n)$ we obtain the injection 
\begin{equation}\label{Eqinj2} H^1(X,Z^{1}_X\otimes V(-n))\hookrightarrow H^1(X,{F_X}_{\ast}(\Omega_X^{1})\otimes V(-n))\end{equation}
for all $n>r_2$, with $r_2:=\mu_{max}(B^{2}_X)$. Combining (\ref{Eqinj1}) and (\ref{Eqinj2})  we obtain the injection

$$H^1(X,B^{1}_X\otimes V(-n))\hookrightarrow H^1(X,{F_X}_{\ast}(\Omega_X^{1})\otimes V(-n))$$
for all $n>max\{r_1,r_2\}$. But the group $H^1(X,{F_X}_{\ast}(\Omega_X^{1})\otimes V(-n))$ is trivial for all $n>n_1$, according to Remark \ref{remOmega}. Now let us set $n_2:=max\{r_1,r_2,n_1\}$, then for all $n>n_2$ we have $H^1(X,B^{1}_X\otimes V(-n))=0$, as required.

\end{proof}

\begin{lemma}\label{lemmaVik2} Assume $dim(X)\geq 3$. Let $V$ be any essentially finite vector bundle over $X$, then there exists a \textit{uniform}  positive integer $n_2$ such that   $H^2(X,V(-n))$ vanishes for all  $n> n_2$.
\end{lemma}

\proof
(See also \cite{BiswHoll}, Lemma 4.12.) By Remark \ref{remCartier} we have the exact sequence 
$$
0\to \mathcal{O}_X \to {F_X}_{\ast}(\mathcal{O}_X)\to B^1_X\to 0$$
that we tensor  by $V(-n)$; then we have  the induced long exact sequence

$$..\to  H^1(X,V(-n)\otimes B^1_X)\to  H^2(X,V(-n))\to  H^2(X,V^p(-np))\to ..$$
According to Lemma \ref{lemmaPreVik2} there exists a uniform integer $n_2$ (independent of $V$) such that for all $n>n_2$ the group $H^1(X,V(-n)\otimes B^1_X)$ vanishes  then we have the injection $H^2(X,V(-n))\hookrightarrow  H^2(X,V^p(-np))$ and, iterating, 
$$H^2(X,V(-n))\hookrightarrow  H^2(X,V^{(p^t)}(-np^t));$$
but $ H^2(X,V^{(p^t)}(-np^t))=0$ for $t>>0$ (cf. Remark \ref{remDecember}) then finally $H^2(X,V(-n))=0$ for all $n>n_2$
and we are done.
\endproof

\section{Theorems}
\label{sec:Theorems}

\begin{notation} Throughout this section  $k$ will be, as before, an algebraically closed field of positive characteristic $p$. Furthermore $W:=W(k)$ will denote the ring of Witt vectors over $k$ and ${R}$ the integral closure of $W$ in the algebraic closure ${\overline{K}}$ of $K:=Frac(W)$.  Till the end of the paper  $X$ will denote a connected,  smooth and projective scheme over ${W}$. Everywhere the subscript ${}_R$ will denote the pull back over $Spec(R)$.
\end{notation}

For  any finite extension $K'$ of $K$ we will denote by $W'$ the integral closure of $W$ in $K'$ (with residue field $k$) and by  $X'$ the fibered product $X\times_W W'$. An essentially finite vector bundle over $X'$ has been defined in \cite{MS2} as a vector bundle $V$ over $X'$ whose restrictions $V_k$ and $V_{{K'}}$ respectively to ${(X')}_k:=X'\times_{W'} k\simeq X_k$ and ${(X')}_{K'}:=X'\times_{W'} K'$ are essentially finite in the usual sense. Let us fix  a $W$-valued point $x\in X(W)$ and let $x_R$ be the induced $R$-valued point on $X_R$. Let $\mathcal{L}$ be the full subcategory of $Coh(X_R)$ (coherent sheaves over $X_R$) whose objects are defined as follows: $V\in Ob(Coh(X_R))$ belongs to $Ob(\mathcal{L})$ if and only if there exists a finite extension   $K'$  of  $K$ and an essentially finite vector bundle $V'$   over $X'$,   such that $V$ is  pull back of  $V'$ over $X_R$. The category $\mathcal{L}$ provided with the fiber functor $x_R^*:\mathcal{L}\to R$-$mod$ is a tannakian lattice as defined by Wedhorn in \cite{Wed}. We denote by $\pi_1(X_R,x_R)$ the affine $R$-group scheme associated to it which is the fundamental group scheme of $X_R$. Now let $x'$ be the section on $X'$ induced by $x$. A principal bundle $E$ over  $X'$, pointed over $x'$,  is called Nori-reduced if $H^0(E,\mathcal{O}_E)={W'}$. 

\begin{theorem}\label{theorem1} Let $X$ be a smooth, connected and projective scheme over ${W}$ of relative dimension $dim(X/{W})\geq 2$. Let $H$ be a relatively very ample line bundle on $X$. Then there exists an integer $d_0$ (depending only on $X$) such that for any $d\geq d_0$, any $Y\in |H^{\otimes d}|$ connected and smooth over ${W}$ and any section $y\in Y({W})$ the homomorphism $\widehat{\varphi}:\pi_1(Y_R,y_R)\to \pi_1(X_R,y_R)$ induced by the closed immersion $\varphi:Y\hookrightarrow X$ is faithfully flat.
\end{theorem}

\begin{proof} As before let $K'$ be any finite extension of $K$, $W'$ the integral closure of $W$ in $K'$ and  $X'$ (resp. $Y'$)  the fibered product $X\times_W W'$ (resp. $Y\times_W W'$). It is sufficient to prove that for any Nori-reduced principal bundle $e:E\to X'$ its restriction to $Y'$, denoted $e_{Y'}:E_{Y'}\to Y'$, is still Nori-reduced. So we are assuming that $H^0(E,\mathcal{O}_E)=H^0(X',e_{\ast}(\mathcal{O}_E))={W'}$. We tensor by $e_{\ast}(\mathcal{O}_E)$ the exact sequence
  
   $$0\to \mathcal{O}_{X'}(-d) \to  \mathcal{O}_{X'} \to \mathcal{O}_{Y'} \to 0$$ 
and we consider the long exact sequence associated   
   
   $$..\to H^0(X',e_{\ast}(\mathcal{O}_{E}))\to  H^0(Y',{e_{Y'}}_{\ast}(\mathcal{O}_{E_{Y'}})) \to  H^1(X',e_{\ast}(\mathcal{O}_{E})(-d)) \to  .. $$ 
by Lemma \ref{lemmaVik1} there exists a uniform $d_0$ (depending only on $X_k$) such that $H^1(X_k,{e_k}_{\ast}(\mathcal{O}_{E_k})(-d))=0$ for all $d\geq d_0$ and where $E_k$ denotes the special fiber of $E$. This implies that $H^1(X',e_{\ast}(\mathcal{O}_{E})(-d))$ vanishes for all $d\geq d_0$. 
Thus $E_{Y'}\to Y'$ is Nori-reduced as $H^0(Y',{e_{Y'}}_{\ast}(\mathcal{O}_{E_{Y'}}))={W'}$ and this concludes the proof.
\end{proof}



\begin{theorem}\label{theorem2} Let $X$ be a smooth, connected and projective scheme over $W$ of relative dimension $dim(X/{W})\geq 3$. Let $H$ be a relatively very ample line bundle on $X$. Then there exists an integer $d_1$ (depending only on $X$) such that for any $d\geq d_1$, any $Y\in |H^{\otimes d}|$ connected and smooth over ${W}$ and any section $y\in Y({W})$ the homomorphism $\widehat{\varphi}:\pi_1(Y_R,y_R)\to \pi_1(X_R,y_R)$ induced by the closed immersion $\varphi:Y \hookrightarrow X$ is an isomorphism.
\end{theorem}

\begin{proof}   According to Theorem \ref{theorem1} we only need to prove that $\widehat{\varphi}$ is a closed immersion. This is equivalent to prove that for any finite extension $K'$ of $K$ and any essentially finite vector bundle $V$ over $Y'$ there exists an essentially finite vector bundle $U$ over $X'$ whose restriction $U_{|{Y'}}=\varphi'^{\ast}(U)$ is isomorphic to $V$ ($W'$, $X'$ and $Y'$  are constructed as in the proof of Theorem \ref{theorem1} and $\varphi':Y'\hookrightarrow X'$ is the morphism induced by $\varphi$). Let us set  $X'_n:=X'\times_{{W}}{W}/p^{n+1}$ and $Y'_n:=Y'\times_{{W}}{W}/p^{n+1}$  for every nonnegative integer $n$, thus in particular $X'_0=X'_k$ and $Y'_0=Y'_k$. Let us denote by $V_n$ the $n$-th  restriction of $V$ to $Y'_n$ and similarly $\varphi'_n:Y'_n\hookrightarrow X'_n$ the $n$-th restriction of $\varphi'$. Now  consider the cartesian diagram 

$$\xymatrix{Y'_1\ar@{^{(}->}[r]& X'_1\\ Y'_0\ar[u]\ar@{^{(}->}[r]& X'_0. \ar[u] }$$
  By \cite{BiswHoll} and \cite{Mehta1} we know that there exists $n_1$ (depending only on $X'_0$) such that for every $d\geq n_1$  and  for every $Y'_0\in |H_k^{\otimes d}|$ the homomorphism $$\widehat{\varphi_0}:\pi_1(Y'_0,y'_k)\to\pi_1(X'_0,y'_k)$$ between the fundamental group schemes of the special fibers of $Y'$ and $X'$, induced by $\varphi'_0:Y'_0\to X'_0$,  is an isomorphism. This implies that there exists an essentially finite vector bundle $U_0$ over $X'_0$ such that ${\varphi'}_0^{\ast}(U_0)\simeq V_0$.  From the exact sequence
  
     $$0\to \mathcal{O}_{X'_0}(-d) \to  \mathcal{O}_{X'_0} \to \mathcal{O}_{Y'_0} \to 0$$ 
   
   we obtain, first tensoring by $\mathcal{E}nd(U_0)$, the long exact sequence 
  
  $$\xymatrix{H^1(X'_0,\mathcal{E}nd(U_0)(-d))\ar[r] & H^1(X'_0,\mathcal{E}nd(U_0))\ar[r]^{\delta_1} & H^1(Y'_0,\mathcal{E}nd({V_0}))\ar[r] & \\ H^2(X'_0,\mathcal{E}nd(U_0)(-d))\ar[r] & H^2(X'_0,\mathcal{E}nd(U_0))\ar[r]^{\delta_2} & H^2(Y'_0,\mathcal{E}nd({V_0}))\ar[r] & ...} $$
  
  By Lemmas \ref{lemmaVik1} and \ref{lemmaVik2} there exists a uniform positive integer $d_1$ such that  $H^1(X'_0,\mathcal{E}nd(U_0)(-d))$ and $H^2(X'_0,\mathcal{E}nd(U_0)(-d))$ vanish for every $d\geq d_1$ thus in particular ${\delta_1}$ is an isomorphism and ${\delta_2}$ is injective. This implies that the induced maps:
  
  $$\delta_1':H^1(X'_0,\mathcal{E}nd(U_0))\otimes_k (p/p^2)\to H^1(Y'_0,\mathcal{E}nd({V_0}))\otimes_k (p/p^2)$$ 
  
  and 
  
  $$\delta_2':H^2(X'_0,\mathcal{E}nd(U_0))\otimes_k (p/p^2) \to H^2(Y'_0,\mathcal{E}nd({V_0}))\otimes_k (p/p^2)$$ 
  
  are respectively an isomorphism and an injection. Now $X'_0\to X'_1$ is a thickening of order one (cf. \cite{Illu}, (8.1.3)) so the obstruction  for the existence of a vector bundle over $X'_1$ whose restriction over $X'_0$ is isomorphic to $U_0$ corresponds to an element $o(U_0)$ of  $H^2(X'_0,\mathcal{E}nd(U_0))\otimes_k (p/p^2)$ (cf. \cite{Illu}, Theorem 8.5.3 or \cite{GroBour}, \S 6, Proposition 3). Then consider $\delta_2'(o(U_0))=o(V_0)\in H^2(Y'_0,\mathcal{E}nd({V_0}))\otimes_k (p/p^2)$. This is zero because clearly $V_1$ extends $V_0$, but $\delta_2'$ is injective hence $o(U_0)=0$ too and $U_0$ can thus be extended to a vector bundle $U_1'$ over $X'_1$. In general ${\varphi'}_1^{\ast}(U_1')$ is not isomorphic to $V_1$, but we know (loc. cit.) that the set of deformations of $V_0$ over $Y'_1$ is an affine space under $H^1(Y'_0,\mathcal{E}nd(V_0))\otimes_k (p/p^2)$. So there exists a unique $t\in H^1(Y'_0,\mathcal{E}nd(V_0))\otimes_k (p/p^2)$ such that ${\varphi'}_1^{\ast}(U_1')=V_1 + t$, then the vector bundle $U_1:=U_1'-(\delta'_1)^{-1}(t)$ over $X'_1$ is such that ${\varphi'}_1^{\ast}(U_1)\simeq V_1$ over $Y'_1$ and of course $U_1$ is still a deformation of $U_0$ as $(\delta'_1)^{-1}(t) \in H^1(X'_0,\mathcal{E}nd(U_0))\otimes_k (p/p^2)$ and the set of deformations of $U_0$ over $X'_1$ is an affine space under $H^1(X'_0,\mathcal{E}nd(U_0))\otimes_k (p/p^2)$.
Now  consider the cartesian diagram 

$$\xymatrix{Y'_2\ar@{^{(}->}[r]& X'_2\\ Y'_1\ar[u]\ar@{^{(}->}[r]& X'_1. \ar[u] }$$

The obstruction for the existence of a vector bundle  $U_2$ over $X'_2$ deforming $U_1$ corresponds to an element $o(U_0)$ of  $H^2(X'_0,\mathcal{E}nd(U_0))\otimes_k (p^2/p^3)$; thus proceeding as in previous step we can find $U_2$ such that its restriction to $Y'_2$ is isomorphic to $V_2$.  It is now clear that for any $n$ we can construct a vector bundle $U_n$ over $X'_n$ extending $U_{n-1}$ over $X'_{n-1}$ whose restriction to $Y'_{n}$ is isomorphic to $V_{n}$. Set $\widetilde{X'}:=\underrightarrow{lim}_{i\in
\mathbb{N} } X'_i$ and $\widetilde{U}:=\underleftarrow{lim}_{i\in
\mathbb{N} } U_i$. Then $\widetilde{U}$ is a vector bundle over $\widetilde{X'}$ and by \cite{EGAIII-1} \S 5.1, we finally obtain a vector bundle $U$ over $X'$ whose restriction to $Y'$ is isomorphic to $V$ and whose special fiber is isomorphic to $U_0$, by construction. It remains to prove that $U$ is essentially finite and this will be clear once we will prove that  its generic fiber $U_{{K'}}$ is essentially finite over $X'_{{K'}}$. Since $V_{{K'}}$ is essentially finite then there exists a principal $G$-bundle $f:T\to Y'_{{K'}}$, for $G$ an \'etale finite ${{K'}}$-group scheme, such that $f^{\ast}(V_{{K'}})\simeq \mathcal{O}_T^{\oplus r}$, where $r:=rk(U_{{K'}})$. By Grothendieck-Lefschetz's  theorem in characteristic 0 we know that there exists a principal $G$-bundle $f':T'\to X'_{{K'}}$ over $X'_{{K'}}$ whose restriciton to $Y'_{{K'}}$ is isomorphic to $f:T\to Y'_{{K'}}$:

$$\xymatrix{T\ar@{^{(}->}[r]\ar[d]_{f}& T'\ar[d]^{f'}\\ Y'_{{{K'}}}\ar@{^{(}->}[r]& X'_{{K'}}.  }$$

 We want to prove that $f'$ trivializes $U_{{K'}}$.  So let us consider the closed immersion $T\hookrightarrow T'$ and the associated short exact sequence 

$$0\to \mathcal{O}_{T'}(-d) \to \mathcal{O}_{T'} \to \mathcal{O}_{T} \to 0$$

that we tensor by $f'^{\ast}(U_{{K'}})$, so that we obtain the following long exact sequence

$$..\to  H^0(T',f'^{\ast}(U_{{K'}})) \to H^0(T,\mathcal{O}_{T})^{\oplus r} \to H^1(T',f'^{\ast}(U_{{K'}})(-d)) \to ..$$

If we prove that $H^1(T',f'^{\ast}(U_{{K'}})(-d))=0$ then we are done: first of all we observe that $U_k$ has zero Chern classes as $V_k$ has zero Chern classes. This implies that $U_{K'}$ has zero Chern classes too. Moreover $U_{K'}$ is semistable because $U_k$ is. Furthermore $char(K')=0$ then  in particular  $H^1(T',f'^{\ast}(U_{{K'}})(-d))=0$ vanishes as desired. 

\end{proof}


\medskip

\scriptsize

\begin{flushright} Marco Antei\\ 

Dept. of Mathematical Sciences,\\Korea Advanced Institute of Science and Technology\\
 Yuseong-gu, Daejeon 305-701, Republic of Korea\\
E-mail:  \texttt{marco.antei@gmail.com}\\
\end{flushright}

\begin{flushright} Vikram B. Mehta\\ 
School of Mathematics,\\ Tata Institute of Fundamental Research,\\ Homi Bhabha Road,
Bombay 400005, India\\
E-mail:  \texttt{vikram@math.tifr.res.in}\\
\end{flushright}

\end{document}